\documentclass[10pt]{article}%
\usepackage{amsmath}
\usepackage{amsfonts}
\usepackage{mathrsfs}
\usepackage{amssymb, color}
\usepackage[linkcolor=black,anchorcolor=black,citecolor=black]{hyperref}
\usepackage{graphicx}
\numberwithin{equation}{section}
\usepackage[body={15.5cm,21cm}, top=3cm]{geometry}%
\providecommand{\U}[1]{\protect\rule{.1in}{.1in}}
\providecommand{\U}[1]{\protect \rule{.1in}{.1in}}
\newtheorem{theorem}{Theorem}[section]

\newtheorem{corollary}[theorem]{Corollary}

\newtheorem{proposition}[theorem]{Proposition}
\newtheorem{remark}[theorem]{Remark}

\newenvironment{proof}[1][Proof]{\noindent \textbf{#1.} }{\  \rule{0.5em}{0.5em}}

\DeclareMathOperator*{\essinf}{ess\,inf}
\begin{document}
	
	\title{Stochastic representation  under $g$-expectation and applications: the discrete time case}
\author{  Miryana Grigorova \thanks{School of Mathematics, University of Leeds, M.R.Grigorova@leeds.ac.uk.}
\and	Hanwu Li\thanks{Research Center for Mathematics and Interdisciplinary Sciences, Shandong University,		lihanwu@sdu.edu.cn.}
}
\date{}
\maketitle

\begin{abstract}
	In this paper, we address the stochastic representation problem  in discrete time under (non-linear) $g$-expectation. We establish  existence and uniqueness of the solution,  as well as a characterization of the solution.  As an application, we investigate a new approach to the optimal stopping problem under $g$-expectation and the related pricing of American options under Knightian uncertainty. Our results are also applied to   a (non-linear) Skorokhod-type obstacle problem. 
\end{abstract}


{\textbf{MSC2010 subject classification}}: 60H10, 60G40

\section{Introduction}
The stochastic representation problem under  linear expectations was first investigated by  Bank and El Karoui \cite{BE} (2004) for the continuous time case, and by Bank and F\"{o}llmer \cite{BF} (2003) for the discrete time case.


For a given real-valued optional process $X=\{X_t\}_{t\in[0,T]}$ (which is required to have certain  regularity properties), the stochastic representation problem (in continuous time) aims at constructing   a unique progressively measurable process $L=\{L_t\}_{t\in[0,T]}$ such that the given process $X$ can be written as:
\begin{displaymath}
	X_t=E_t[\int_t^T f(s,\sup_{t\leq v\leq s}L_v)ds], \ \ 0\leq t\leq T,
\end{displaymath}
where $f=f(t,l)$ is a given function, assumed to be continuous and strictly decreasing  with respect to $l$ (from $+\infty$ to $-\infty$) and $E_t[\cdot]$ denotes the (linear) conditional expectation with respect to the information available at time $t$.  Bank and El Karoui \cite{BE} show that there exists a unique solution $L$ to the stochastic representation problem. Moreover,   the  solution is  characterized by the following: for every stopping time $\sigma<T$,
\begin{equation}\label{sr6}
	L_\sigma=\essinf_{\tau \in\mathcal{T}_\sigma}l_{\sigma, \tau},\quad  P \textrm{-a.s.},
\end{equation}
where $\mathcal{T}_\sigma$ is the set of all stopping times $\tau$ such that  $\tau>\sigma$ on the set $\{\sigma<T\}$, $P$-a.s. and $l_{\sigma,\tau}$ is the unique $\mathcal{F}_\sigma$-measurable random variable satisfying 
\begin{displaymath}
	E_\sigma[X_\sigma-X_\tau]=E_\sigma[\int_\sigma^\tau f(t,l_{\sigma,\tau})dt].
\end{displaymath}
The stochastic representation results have been successfully   applied to  various stochastic control  problems in mathematical finance and  mathematical economics, such as  optimal consumption choice with Hindy-Huang-Kreps-type preferences  (see  \cite{BR}, \cite{FRS}),   irreversible investment  (see \cite{CF}, \cite{CFR}, \cite{RS}),  dynamic allocation problems (see \cite{EK}), or  a variant of Skorokhod's obstacle problem (see \cite{MW}). Roughly speaking,  finding the optimal consumption plan with intertemporal substitution, the base capacity policy of the irreversible investment problem and the solution to a certain obstacle problem of the Skorokhod type amounts to finding  the solution of a specific stochastic representation problem. More recently,  \cite{BB} extend further the framework for validity and applications of the stochastic representation problem  by using some fine notions and techniques from the general theory of processes. 

Since,  in the above framework, stochastic representation is considered under one probability measure $P$, it cannot be applied to address financial or economic problems involving ambiguity/Knightian uncertainty. Uncertainty typically leads to non-linearity of the ``expectation" operators. 
It is well-known that the (non-linear) $g$-expectation (cf.  Peng \cite{P97})   is  a powerful tool to study problems with ambiguity. In this paper, we are interested in  the stochastic representation problem  under $g$-expectation; formally, this amounts to  replacing  the classical conditional expectation $E_t[\cdot]$ in the formulation of the problem by the conditional $g$-expectation $\mathcal{E}_t[\cdot]$. It is worth pointing out that the construction of the solution to the representation theorem studied by Bank and El Karoui \cite{BE} heavily depends on the linearity of the classical conditional expectation which means that their  construction method is not effective for the non-linear $g$-expectation case. In the current work, we focus on the   non-linear representation problem  in  discrete time. In order to prove the existence, we apply the method of backward induction. The uniqueness is proved by using the fact that the function $f$ is strictly decreasing (in the last component) and the property of strict monotonicity of the  $g$-expectation.   Unlike  the continuous time case (cf. \cite{BE}), we do not need to establish a characterization of the solution $(L_t)$ analogous to \eqref{sr6} to obtain the uniqueness. However, a non-linear analogue of this  characterization  still holds true in our framework. We provide moreover a construction of  a stopping   stopping time $\tau^*_t$ which is optimal, in the sense that $L_{t}=l_{t,\tau^*_t}$. It is worth pointing out that the conditions on the driver $g$ to guarantee the  existence and uniqueness result are weaker than those made  to guarantee the characterization  of the solution. 

 The second part of this paper provides several applications of the stochastic representation problem under $g$-expectation, namely to optimal stopping, to optimal exercise of  American put options under Knightian uncertainty, and to a variant of Skorokhod's obstacle problem.

  It is well known that the stochastic representation problem under \emph{linear} conditional expectations has    strong connections with the (classical) optimal stopping problem (cf. \cite{BF}). It  provides an alternative approach to the celebrated Snell envelope approach to optimal stopping, with fruitful applications in pricing of American options.  In this alternative approach, the solution  $L$ of the stochastic representation for the given reward (or pay-off) process $X$  takes over the role  of the Snell envelope of $X$.  When applied to American options,  this approach allows to  find a \textit{universal} process  \textit{not depending  on the strike price} through which   optimal exercise times can be characterized. 
 For the \emph{non-linear} case, the Snell envelope approach to  optimal stopping under $g$-expectation is well-studied (see, e.g.,  \cite{CR}, \cite{BY},  \cite{GIOQ} for the continuous time case, or \cite{GQ} for the discrete time case). The  stochastic representation results from the first part of our paper  give a new  approach to the non-linear optimal stopping problem under $g$-expectations. This approach is then  applied to derive an optimality criterion  for   American put options under Knightian uncertainty in terms of a universal process independent of the strike price $k$ of the option. In the third application,  the solution of our stochastic representation problem  is used to solve   a variant of Skorokhod's obstacle problem. More specifically, we show that the increasing process $\eta$ from the Skorokhod-type condition in this problem   coincides with  the running supremum of the solution $L$ to the stochastic representation problem for the obstacle process $X$. 

The paper is organized as follows. In Section 2, we first formulate the non-linear stochastic representation problem in discrete time  under $g$-expectations  and establish the existence and uniqueness result, as well as the characterization of the solution.  In Section 3, we present the three   applications: to  optimal stopping, to  the class of American put options with strike prices $k>0$, and to an  obstacle problem of Skorokhod type.

\section{The non-linear stochastic representation problem in discrete time: formulation, existence and uniqueness}

We place ourselves on the canonical space.
Let $\Omega=C_0^d([0,\infty))$ be the space of all continuous, $\mathbb{R}^d$-valued functions on $[0,\infty)$, equipped with the distance:
\begin{displaymath}
	d(\omega^1,\omega^2)=\sum_{n=1}^{\infty}\frac{1}{2^n}\max_{0\leq t\leq n}(|\omega^1_t-\omega^2_t|\wedge 1).
\end{displaymath}
The $\sigma$-algebra is the Borel $\sigma$-algebra.  Let  $P$ be the Wiener measure, under which the canonical process $B$ is a $d$-dimensional Brownian motion. Let $\mathbb{F}=(\mathcal{F}_t)$ be the filtration generated by the Brownian motion $B$. Let  $N\in\mathbb{N}$  be a fixed terminal horizon.
We denote by $L^2(\mathcal{F}_N)$ the space of all $\mathcal{F}_N$-measurable and square-integrable random variables.
In the sequel, the notation  $g:[0,N]\times \Omega\times \mathbb{R}^d\rightarrow \mathbb{R}$  will stand for a driver satisfying the following standard assumptions (unless specified otherwise):
\begin{description}
	\item[(i)] $(g(t,\omega,z))_{t\in[0,N]}$ is progressively measurable and for any $z\in \mathbb{R}^d$,
	\begin{displaymath}
	E[\int_0^N |g(t,z)|^2dt]<\infty;
	\end{displaymath}
	\item[(ii)] There exists a constant $K>0$, such that
	\begin{displaymath}
	|g(t,\omega,z)-g(t,\omega,z')|\leq K|z-z'|;
	\end{displaymath}
	\item[(iii)] For any $(s,\omega)$, $g(s,\omega,0)=0$.
\end{description}
By  a well-known  result of Pardoux and Peng \cite{PP90},  for any terminal condition $X\in L^2(\mathcal{F}_N)$, the Backward SDE
\begin{displaymath}
Y_t=X+\int_t^Ng(s,Z_s)ds-\int_t^N Z_sdB_s,
\end{displaymath}
has a unique adapted solution $(Y,Z)$.   The non-linear expectation operator, induced by a BSDE of the above form,  is known as  conditional $g$-expectation, and    is defined by
\begin{displaymath}
\mathcal{E}_t[X]:=Y_t.
\end{displaymath}
Some of the main properties of the conditional  $g$-expectation are recalled  in the following proposition:
\begin{proposition}
Under the above assumptions on the driver $g$, 	the conditional $g$-expectation satisfies the following properties:
	\begin{description}
		\item[(1)] (monotonicity and strict monotonicity) If $X\leq Y$, then $\mathcal{E}_t[X]\leq \mathcal{E}_t[Y]$. \\
		If, in addition $P(X<Y)>0$, then  $\mathcal{E}_t[X]<\mathcal{E}_t[Y]$;
		\item[(2)] (translation invariance) If $Z\in L^2(\mathcal{F}_t)$, then for all $X\in L^2(\mathcal{F}_N)$, $\mathcal{E}_t[X+Z]=\mathcal{E}_t[X]+Z$;
		\item[(3)] (tower property) For any $0\leq s\leq t\leq T$, $\mathcal{E}_s[\mathcal{E}_t[X]]=\mathcal{E}_s[X]$;
		\item[(4)] (zero-one law) For an event $A\in\mathcal{F}_t$, it holds $\mathcal{E}_t[XI_A+YI_{A^c}]=\mathcal{E}_t[X]I_A+\mathcal{E}_t[Y]I_{A^c}$.
		 \item[(5)] (monotone convergence) For a monotone sequence $\{X_n\}_{n\in\mathbb{N}}\subset L^2(\mathcal{F}_N)$ such that $X_n\uparrow(\downarrow) X$, where $X\in L^2(\mathcal{F}_N)$, we have $\mathcal{E}_t[X_n]\uparrow (\downarrow) \mathcal{E}_t[X]$.
	\end{description}
\end{proposition}

	    Let $X=\{X_t\}_{t=0}^N$ be a given real-valued, adapted and square-integrable process  and let   $f:\Omega\times\{0,1,\cdots,N\}\times \mathbb{R}\rightarrow \mathbb{R}$ be a given function satisfying the following two conditions:
\begin{description}
	\item[(1)] For each $\omega\in \Omega$ and each $t=0,1,\cdots,N$, the function $f(\omega,t,\cdot):\mathbb{R}\rightarrow\mathbb{R}$ is continuous and strictly decreasing from $+\infty$ to $-\infty$;
	\item[(2)] For any $l\in\mathbb{R}$, the process $f(\cdot,\cdot,l):\Omega\times\{0,1,\cdots,N\}\rightarrow\mathbb{R}$ is adapted with
	\begin{displaymath}
		E[\sum_{t=0}^{N}|f(t,l)|^2]<\infty.
	\end{displaymath}
\end{description}

The \textit{non-linear stochastic representation problem in discrete time} is formulated as follows: \\
Find an \emph{adapted} process $L=\{L_t\}_{t=0,1,\cdots,N}$ such that  $\sum_{u=t}^{N}f(u,\max_{t\leq v\leq u}L_v)$ is square-integrable for  all $t=0,1,\cdots,N$, and  such that the following equation holds:
\begin{equation}\label{sr4}
X_t=\mathcal{E}_t[\sum_{u=t}^{N}f(u,\max_{t\leq v\leq u}L_v)],\textrm{ for all }t=0,1,\cdots,N.
\end{equation}


A process $(L_t)$ satisfying these properties will be called \emph{a solution} to the non-linear stochastic representation problem  \eqref{sr4}. We now  state and prove the main result of this section.

\begin{theorem}[Existence and Uniqueness]\label{srt5}
	Under the Assumptions (1)-(2) on the function 
$f$ and Assumptions (i)-(iii) on the driver $g$, there \emph{exists a unique} solution $(L_t)$ to the non-linear stochastic representation problem \eqref{sr4}.
\end{theorem}

\begin{proof}
	We first prove the uniqueness. Suppose that $L^1$ and $L^2$ are are two solutions of the stochastic representation problem  \eqref{sr4}. We have to show that $L_t^1=L^2_t$, for all $t=0,1,\ldots, N$. We proceed by backward induction.  It is easy to check that $L^1_N=L^2_N=f^{-1}(N,X_N)$.  Let $t\in\{0,\cdots,N\}$.   Assume that for all $k=t+1,\cdots,N$, we have shown $L_k^1=L_k^2=:L_k$. Let us show that $L_t^1=L_t^2$. Set $A=\{L_t^1<L_t^2\}$ and $A'=\{L_t^1>L_t^2\}$. Suppose, by way of contradiction, that $P(A)>0$. Since $A\in\mathcal{F}_t$,  we have, for $i=1,2$,
	\begin{align*}
		X_tI_A=&I_A\mathcal{E}_t[f(t,L_t^i)+\sum_{k=t+1}^N f(k,L_k^i\vee (\max_{t+1\leq v\leq k}L_v))]\\
		=&\mathcal{E}_t[f(t,L_t^i)I_A+\sum_{k=t+1}^N I_Af(k,L_k^i\vee (\max_{t+1\leq v\leq k}L_v))],
	\end{align*}
	where we have used the zero-one law for conditional $g$-expectation (property (4)).
	On the set $A$, since $f$ is strictly decreasing, we have $f(t,L_t^1)>f(t,L_t^2)$ and
	\begin{displaymath}
		\sum_{k=t+1}^N f(k,L_k^1\vee (\max_{t+1\leq v\leq k}L_v))\geq \sum_{k=t+1}^N f(k,L_k^2\vee (\max_{t+1\leq v\leq k}L_v)).
	\end{displaymath}
	By the strict comparison theorem for conditional $g$-expectations, we get that, on the set $A$,
	\begin{align*}
		&\mathcal{E}_t[f(t,L_t^1)I_A+\sum_{k=t+1}^N I_Af(k,L_k^1\vee (\max_{t+1\leq v\leq k}L_v))]\\
		>	&\mathcal{E}_t[f(t,L_t^2)I_A+\sum_{k=t+1}^N I_Af(k,L_k^2\vee (\max_{t+1\leq v\leq k}L_v))],
	\end{align*} which is  a contradiction.
	We conclude that $P(A)=0$. By interchanging the roles of $L^1$ and $L^2$ in the above reasoning, we get that $P(A')=0$. Hence, the uniqueness is shown.
	
	
 		We now show the existence. We proceed by backward induction. It is easy to check that $L_{N}$ defined by $L_{N}=f^{-1}(N,X_{N})$ is a solution to the stochastic representation problem  at the terminal time $N$ and that $f(N,L_N)$ is square-integrable. Let $t\in\{0,\ldots, N\}$.   Suppose  that we have shown the existence of an adapted process  $\{L_k\}_{k=t+1,\cdots,N}$ such that $\sum_{u=k}^{N}f(u,\max_{k\leq v\leq u}L_v)$ is square-integrable, for all $k=t+1,\cdots,N$ and such that
	\begin{displaymath}
		X_k=\mathcal{E}_k[\sum_{u=k}^{N}f(u,\max_{k\leq v\leq u}L_v)], \textrm{ for all } k=t+1,\cdots,N.
	\end{displaymath}
	For $k=t$, we set $\mathcal{H}_t:=\{\xi|\xi \textrm{ is }\mathcal{F}_t\textrm{-measurable}, \widetilde{f}(t,N,\xi) \textrm{ is square-integrable and }  \mathcal{E}_t[\widetilde{f}(t,N,\xi)]\leq X_t\}$, where
	\begin{displaymath}
		\widetilde{f}(t,N,\xi)=f(t,\xi)+\sum_{u=t+1}^{N}f(u,\xi\vee(\max_{t+1\leq v\leq u}L_v)).
	\end{displaymath}
	Since for each fixed $t,\omega$, $f(t,\omega,\cdot)$ is strictly decreasing from $+\infty$ to $-\infty$, by the monotone convergence theorem, we have
	\begin{displaymath}
		\lim_{M\rightarrow\infty}\mathcal{E}_t[\widetilde{f}(t,N,M)]=-\infty.
	\end{displaymath}
	Therefore, the set  $\mathcal{H}_t$ is non-empty. We define
	\begin{displaymath}
		L_t:=\essinf_{\xi\in\mathcal{H}_t}\xi.
	\end{displaymath}
	We will show that $L_t$ is a solution to the representation problem at time $t$.
	For this purpose, we first show that the set $\mathcal{H}_t$ is downward directed. Let $\xi^i\in\mathcal{H}_t$, for $i=1,2$. Set
	\begin{displaymath}
		\xi=\xi^1 I_B+\xi^2 I_{B^c},
	\end{displaymath}
	where $B=\{\xi^1\leq \xi^2\}\in\mathcal{F}_t$. It is easy to check that
	\begin{displaymath}
		\mathcal{E}_t[\widetilde{f}(t,N,\xi)]=\mathcal{E}_t[\widetilde{f}(t,N,\xi^1)]I_B+\mathcal{E}_t[\widetilde{f}(t,N,\xi^2)]I_{B^c}\leq X_t,
	\end{displaymath}
	which yields that $\xi\in\mathcal{H}_t$. Hence, the set $\mathcal{H}_t$ is downward directed.  Therefore, there exists a decreasing sequence $\{\xi_n\}\subset \mathcal{H}_t$ such that $L_t=\lim_{n\rightarrow\infty}\xi_n$. By the monotone convergence theorem, we deduce that
	\begin{displaymath}
		\mathcal{E}_t[\sum_{u=t}^{N}f(u,\max_{t\leq v\leq u}L_v)]=\mathcal{E}_t[\widetilde{f}(t,N,L_t)]=\lim_{n\rightarrow\infty}\mathcal{E}_t[\widetilde{f}(t,N,\xi_n)]\leq X_t,
	\end{displaymath}
	which implies that $L_t\in\mathcal{H}_t$. Set $C=\{\mathcal{E}_t[\widetilde{f}(t,N,L_t)]<X_t\}\in\mathcal{F}_t$. In order to conclude, it is sufficient to show that $P(C)=0$.  Suppose, by way of contradiction, that,  $P(C)=\varepsilon>0$.   For each $n\in\mathbb{N}$,  we define
	\begin{displaymath}
		\zeta_n=L_t I_{C^c}+(L_t-\frac{1}{n})I_C.
	\end{displaymath}
	It is easy to check that $\zeta_n\uparrow L_t$ and
	\begin{displaymath}
		\mathcal{E}_t[\widetilde{f}(t,N,\zeta_n)]I_C\downarrow \mathcal{E}_t[\widetilde{f}(t,N,L_t)]I_C<X_tI_C.
	\end{displaymath}
	By Lusin's theorem, 
there exist some $\mathcal{F}_t$-measurable open sets $\{O^\varepsilon_n\}_{n=1}^\infty$ and $O^\varepsilon$ with $P(O^\varepsilon_n)\leq \frac{\varepsilon}{2^{n+2}}$ and $P(O^\varepsilon)\leq \frac{\varepsilon}{8}$, such that $\mathcal{E}_t[\widetilde{f}(t,N,\zeta_n)]I_C I_{(O^\varepsilon_n)^c}$ and $\mathcal{E}_t[\widetilde{f}(t,N,L_t)]I_CI_{(O^\varepsilon)^c}$ are continuous. Set $O=(\cup_{n=1}^\infty O^\varepsilon_n)\cup O^\varepsilon$. It is easy to check that $P(O)\leq \frac{3}{8}\varepsilon$ and $\mathcal{E}_t[\widetilde{f}(t,N,\zeta_n)]I_C I_{O^c}$ and $\mathcal{E}_t[\widetilde{f}(t,N,L_t)]I_C I_{O^c}$ are continuous. By Dini's theorem, $\mathcal{E}_t[\widetilde{f}(t,N,\zeta_n)]I_C I_{O^c}$ converges to $\mathcal{E}_t[\widetilde{f}(t,N,L_t)]I_C I_{O^c}$ uniformly. Then there exists some $M$ independent of $\omega$, such that for any $n\geq M$, $\mathcal{E}_t[\widetilde{f}(t,N,\zeta_n)]I_C I_{O^c}\leq X_tI_C I_{O^c}$. Now let
	\begin{displaymath}
		\widetilde{\zeta}_n=L_t I_{C^c\cup O}+(L_t-\frac{1}{n})I_{C\cap O^c}=L_t I_{C^c\cup O}+\zeta_nI_{C\cap O^c}.
	\end{displaymath}
	It is easy to check that for $n\geq M$,
    	\begin{displaymath}
    \mathcal{E}_t[\widetilde{f}(t,N,\widetilde{\zeta}_n)]=\mathcal{E}_t[\widetilde{f}(t,N,L_t)]I_{C^c \cup O}+\mathcal{E}_t[\widetilde{f}(t,N,\zeta_n)]I_{C\cap O^c}\leq X_t,
    \end{displaymath}
    which implies that $\widetilde{\zeta}_n\in \mathcal{H}_t$. We claim that $P(C\cap O^c)>0$, which leads to a contradiction with the fact that $L_t$ is the essential infimum of $\mathcal{H}_t$. To show that $P(C\cap O^c)>0$, we notice that, if $P(C\cap O^c)=0$, then
    \begin{displaymath}
    	P(C\cup O^c)=P(C)+P(O^c)\geq \varepsilon+1-\frac{3}{8}\varepsilon>1,
    \end{displaymath}
	which is impossible; hence, the claim holds and this completes the proof.
\end{proof}

\begin{remark}\label{srr1}
	Consider a non-linear operator $\mathcal{E}_{t,N}:L^2(\mathcal{F}_N)\rightarrow L^2(\mathcal{F}_t)$ satisfying the following property
	\begin{description}
		\item[(I)] For any $\xi,\eta\in L^2(\mathcal{F}_N)$ with $\xi\leq \eta$, then we have $\mathcal{E}_{t,N}[\xi]\leq \mathcal{E}_{t,N}[\eta]$. Furthermore, if $P(\xi<\eta)>0$, then $\mathcal{E}_{t,N}[\xi]<\mathcal{E}_{t,N}[\eta]$;
		\item[(II)] For any $\{\xi_n\}\subset L^2(\mathcal{F}_N)$ such that $\xi_n\uparrow(\downarrow) \xi$, then we have $\mathcal{E}_{t,N}[\xi_n]\uparrow(\downarrow)\mathcal{E}_{t,N}[\xi]$;
		\item[(III)]  For any $\mathcal{F}_t$-measurable partition $\{A_n\}_{n=1}^M$ and $\{\xi_n\}_{n=1}^M\subset L^2(\mathcal{F}_N)$, we have $\mathcal{E}_{t,N}[\sum_{n=1}^M\xi_n I_{A_n}]=\sum_{n=1}^M\mathcal{E}_{t,N}[\xi_n]I_{A_n}$.
	\end{description}
	By a similar analysis to that of the proof of Theorem \ref{srt5}, it can be shown the stochastic representation problem with  $\mathcal{E}_{t,N}[\cdot]$ satisfying the above properties (I),(II), and (III), has a unique solution. This applies, in particular,  to the following two examples:
	\begin{description}
		\item[(a)] $\mathcal{E}_{t,N}[\xi]:=\mathcal{E}_{t,N}^g[\xi]=Y_t^{N,\xi}$, where $(Y^{N,\xi},Z^{N,\xi})$ is the solution of the following BSDE:
		\begin{displaymath}
		Y^{N,\xi}_t=\xi+\int_t^N g(s,Y_s^{N,\xi},Z_s^{N,\xi})ds-\int_t^T Z_s^{N,\xi}dB_s.
		\end{displaymath}
		Here, $g:[0,N]\times\Omega\times \mathbb{R}\times \mathbb{R}^d\rightarrow\mathbb{R}$ is a standard driver satisfying the following condition
		\begin{description}
			\item[(i')] For each fixed $y\in\mathbb{R}$ and $z\in\mathbb{R}^d$, $(g(t,\omega,y,z))_{t\in[0,N]}$ is progressively measurable and ,
			\begin{displaymath}
			E[\int_0^N |g(t,y,z)|^2dt]<\infty;
			\end{displaymath}
			\item[(ii')] There exists a constant $K>0$, such that
			\begin{displaymath}
			|g(t,\omega,y,z)-g(t,\omega,y',z')|\leq L(|y-y'|+|z-z'|).
			\end{displaymath}
		\end{description}
		\item[(b)] $\mathcal{E}_{t,N}[\xi]:=\alpha_t\mathcal{E}^{g}_{t,N}[\xi]+(1-\alpha_t)\mathcal{E}^{-g}_{t,N}[\xi]$, where
		$\alpha$ is a given adapted process  taking values in $[0,1]$. In this case, $\mathcal{E}_{t,N}[\cdot]$ can be seen as an extension of the alpha-maxmin conditional expectation (cf., e.g.,  \cite{BLR}).
	\end{description}
\end{remark}

\begin{remark}
	In some applications,  we need to consider the stochastic representation problem in a slightly different formulation, where equation \eqref{sr4} is replaced by the following equation: 
	\begin{displaymath}
		X_t=\mathcal{E}_t[\sum_{u=t}^{N-1}f(u,\max_{t\leq v\leq u}L_v)+X_N].
	\end{displaymath}
	Here, and in the sequel, we use the following convention: if $s<t$, for any process $h$, we define $\sum_{u=t}^s h(u)=0$.  
	By similar arguments to those of the proof of  Theorem \ref{srt5}, we can show that there exists a unique adapted solution $L=\{L_t\}_{t=0,1,\cdots,N-1}$ to this  problem.
\end{remark}

We now  establish a characterization of the solution $L$ to the stochastic representation problem \eqref{sr4}. To this purpose, we define the following sets of stopping times:
 $$\mathcal{T}_{0,N}=\{\tau|\tau\textrm{ is a stopping time taking values a.s. in } \{0,1,\cdots,N\}\},$$ 
 $$\mathcal{T}_\sigma=\{\tau\in\mathcal{T}|\tau>\sigma  \quad a.s. \textrm{ on } \{\sigma<N\}\}, \textrm{ where } \sigma\in\mathcal{T}_{0,N}.$$

\begin{proposition}\label{srp1}
	Under the Assumptions (i)-(iii) on the driver $g$ and (1)-(2) on the function $f$, the solution $L$ to the  stochastic representation problem \eqref{sr4} satisfies: For any stopping time  $\sigma\in\mathcal{{T}}_{0,N-1}$,
	\begin{equation}\label{sr}
		L_\sigma=\essinf_{\tau \in\mathcal{T}_\sigma}l_{\sigma,\tau},\ \ \textrm{a.s.,}
	\end{equation}
	where $l_{\sigma,\tau}$ is the unique $\mathcal{F}_\sigma$-measurable solution of the following equation
		\begin{equation}\label{sr1}
	X_\sigma=\mathcal{E}_\sigma[\sum_{u=\sigma}^{\tau-1}f(u,l_{\sigma,\tau})+X_\tau].
	\end{equation}
\end{proposition}

\begin{proof}
	Preliminary Step. By modifying the proof of Theorem \ref{srt5}, we can show that there exists a unique solution $l_{\sigma,\tau}$ to Equation \eqref{sr1}. Thus, it remains to  prove  \eqref{sr}.
	
	Step 1. Let  $\sigma\in\mathcal{{T}}_{0,N-1}$ be a given stopping time and let $\tau\in\mathcal{T}_\sigma$. By Equation \eqref{sr4} and the decreasing property of $f$, we  get
	\begin{align*}
		X_\sigma&=\mathcal{E}_\sigma[\sum_{u=\sigma}^{\tau-1}f(u,\sup_{\sigma\leq v\leq u}L_v)+\mathcal{E}_\tau[\sum_{u=\tau}^{N}f(u,\sup_{\sigma\leq v\leq u}L_v)]]\\
		&\leq \mathcal{E}_\sigma[\sum_{u=\sigma}^{\tau-1}f(u,L_\sigma)+\mathcal{E}_\tau[\sum_{u=\tau}^{N}f(u,\sup_{\tau\leq v\leq u}L_v)]]\\
		&=\mathcal{E}_\sigma[\sum_{u=\sigma}^{\tau-1}f(u,L_\sigma)+X_\tau].
	\end{align*}
	By Equation \eqref{sr1}, it follows that
	\begin{displaymath}
		\mathcal{E}_\sigma[\sum_{u=\sigma}^{\tau-1}f(u,l_{\sigma,\tau})+X_\tau]\leq \mathcal{E}_\sigma[\sum_{u=\sigma}^{\tau-1}f(u,L_\sigma)+X_\tau].
	\end{displaymath}
	We set $A=\{l_{\sigma,\tau}<L_\sigma\}\in\mathcal{F}_\sigma$. We claim that $P(A)=0$. Suppose, by way of contradiction, that $P(A)>0$. By the strictly decreasing property of $f$, we have  that
	\begin{displaymath}
	\mathcal{E}_\sigma[\sum_{u=\sigma}^{\tau-1}f(u,l_{\sigma,\tau})+X_\tau]I_A> \mathcal{E}_\sigma[\sum_{u=\sigma}^{\tau-1}f(u,L_\sigma)+X_\tau]I_A,
	\end{displaymath}
	which is a contradiction. We deduce that $P(A)=0$, that is,  $L_\sigma\leq l_{\sigma,\tau}$ a.s. As $\tau$ is arbitrary in $\mathcal{T}_\sigma$, we get
	\begin{displaymath}
		L_\sigma\leq \essinf_{\tau \in\mathcal{T}_\sigma}l_{\sigma,\tau}.
	\end{displaymath}
	
	Step 2. We now show the converse inequality. For each fixed $\sigma\in\mathcal{T}_{0,N-1}$, and for each $n\in\mathbb{N^*}$, consider the following stopping time
	\begin{displaymath}
		\tau^n:=\inf\{t\geq \sigma|\sup_{\sigma\leq v\leq t}L_v>L_\sigma^n\}\wedge N,
	\end{displaymath}
	where
	\begin{displaymath}
		L_\sigma^n:=(L_\sigma+\frac{1}{n})I_{\{L_\sigma>-\infty\}}-nI_{\{L_\sigma=-\infty\}}.
	\end{displaymath}
	It is easy to check that $\tau^n\in\mathcal{T}_\sigma$. Besides, note that on the set $\{\tau^n<N\}$, we have $L_{\tau^n}=\sup_{\sigma\leq v\leq \tau^n}L_v$, which yields that for any $t\in\{\tau^n,\tau^n+1,\cdots,N\}$
	\begin{displaymath}
		\sup_{\sigma\leq v\leq t}L_v=\sup_{\tau^n\leq v\leq t}L_v.
	\end{displaymath}
	Therefore, we obtain that
	\begin{align*}
		X_\sigma&=\mathcal{E}_\sigma[\sum_{u=\sigma}^{\tau^n-1}f(u,\sup_{\sigma\leq v\leq u}\L_v)+\mathcal{E}_{\tau^n}[\sum_{u=\tau^n}^{N}f(u,\sup_{\tau^n\leq v\leq u}L_v)]]\\
		&\geq \mathcal{E}_\sigma[\sum_{u=\sigma}^{\tau^n-1}f(u,L_\sigma^n)+X_{\tau^n}].
	\end{align*}
	Combining with Equation \eqref{sr1}, it follows that
	\begin{displaymath}
		\mathcal{E}_\sigma[\sum_{u=\sigma}^{\tau^n-1}f(u,l_{\sigma,\tau^n})+X_{\tau^n}]\geq \mathcal{E}_\sigma[\sum_{u=\sigma}^{\tau^n-1}f(u,L_\sigma^n)+X_{\tau^n}].
	\end{displaymath}
	Similar analysis to that of Step 1 shows that
	\begin{displaymath}
		L_\sigma^n\geq l_{\sigma,\tau^n}\geq \essinf_{\tau \in\mathcal{T}_\sigma}l_{\sigma,\tau}.
	\end{displaymath}
	Letting $n\rightarrow\infty$, we get the desired result.
\end{proof}

The following proposition establishes that an optimal stopping time exists.
\begin{proposition}\label{srp2}
	For any $k=2,\cdots, N$, set
	\begin{displaymath}
	\tau_{N-k}^*(\omega)=\begin{cases}
	N-k+1, &\omega \in \{L_{N-k}<L_{N-k+1}\};\\
	N-k+i, &\omega \in \{\max_{j=1,\cdots,i-1} L_{N-k+j}\leq L_{N-k}<L_{N-k+i}\}, i=2,\cdots,k-1;\\
	N, &\omega \in\{\max_{j=1,\cdots,k-1} L_{N-k+j}\leq L_{N-k}\}.
	\end{cases}
	\end{displaymath}
	And let $\tau_{N-1}^*=N$. For each $t=0,1,\cdots,N-1$, the stopping times $\tau^*_t$ is optimal in the sense that
	\begin{displaymath}
		L_t=\essinf_{\tau\in\mathcal{T}_t}l_{t,\tau}=l_{t,\tau_t^*}.
	\end{displaymath}
\end{proposition}

\begin{proof}
	The result is trivial for the case when $t=N-1$. For the other cases, it is sufficient to prove that $P(L_t<l_{t,\tau_t^*})=0$. By the definition of $\tau_t^*$, we can check that
	\begin{displaymath}
		\max_{t\leq v\leq u}L_v(\omega)=\begin{cases}
		\max_{\tau_t^*\leq v\leq u}L_v(\omega), &\omega\in\{\tau^*_t\leq u\},\\
		L_t(\omega), &\omega\in\{u<\tau^*_t\}.
		\end{cases}
	\end{displaymath}
	Therefore, we have
	\begin{align*}
		\mathcal{E}_t[\sum_{u=t}^{\tau_t^*-1}f(u,l_{t,\tau^*_t})+X_{\tau^*_t}]=X_t=&\mathcal{E}_t[\sum_{u=t}^{\tau_t^*-1}f(u,\max_{t\leq v\leq u}L_v)+\mathcal{E}_{\tau^*_t}[\sum_{u=\tau^*_t}^{N}f(u,\max_{t\leq v\leq u}L_v)]]\\
		=&\mathcal{E}_t[\sum_{u=t}^{\tau_t^*-1}f(u,L_t)+\mathcal{E}_{\tau^*_t}[\sum_{u=\tau^*_t}^{N}f(u,\max_{\tau^*_t\leq v\leq u}L_v)]]\\
		=&\mathcal{E}_t[\sum_{u=t}^{\tau_t^*-1}f(u,L_t)+X_{\tau^*_t}].
	\end{align*}
	By a similar analysis as in the proof of Proposition \ref{srp1}, we finally have $P(L_t<l_{t,\tau_t^*})=0$. Hence, the result follows.
\end{proof}

\begin{remark}
	Modifying the proof slightly, similar results still hold (e.g., existence and uniqueness, characterization) if, instead of being strictly decreasing, $f(t,\omega,\cdot)$ is\emph{ strictly increasing} from $-\infty$ to $+\infty$ for each fixed $t$ and $\omega$.
\end{remark}


\section{Applications}

In this section, we present some applications of the stochastic representation problem under $g$-expectation. Throughout this section, we assume that the driver $g$ satisfies conditions (i)-(iii).

\subsection{Optimal stopping under $g$-expectation}

We present a  new approach to the non-linear optimal stopping problem in discrete time. 
This approach is based on the stochastic representation of the given  reward process $X$, established in the previous section.  
This approach can be seen as a non-linear analogue of the  approach presented by  Bank and Follmer \cite{BF} in the linear case.

The following theorem provides a  level-crossing principle  and an optimality criterion for stopping times.

\begin{theorem}(Level-crossing principle and optimality criterion)\label{srt1}
	Let $X=\{X_n\}_{n=0,1,\cdots,N}$ be an adapted and square-integrable sequence and $L=\{L_t\}_{t=0,1,\cdots,N-1}$ be the solution of the following backward equation
	\begin{displaymath}
		X_t=\mathcal{E}_t[\sum_{u=t}^{N-1}\max_{t\leq v\leq u} L_v+X_N].
	\end{displaymath}
	Then, the level-passage times
	\begin{displaymath}
		\underline{\tau}:=\min\{v\geq 0|L_v\geq 0\}\wedge N \textrm{ and } \bar{\tau}:=\min\{v\geq 0|L_v>0\}\wedge N
	\end{displaymath}
	are optimal for the problem
	\begin{displaymath}
		V=\sup_{\tau\in \mathcal{T}_{0,N}}\mathcal{E}[X_\tau].
	\end{displaymath}
	
	Furthermore, if $\tau^*\in\mathcal{T}_{0,N}$ satisfies
	\begin{equation}\label{sr9}
		\underline{\tau}\leq \tau^*\leq \bar{\tau},\textrm{ and } \max_{0\leq v\leq \tau^*}L_v=L_{\tau^*},
	\end{equation}
	then, $\tau^*$ is  an optimal stopping time.
\end{theorem}

\begin{proof}
	For any $\tau\in\mathcal{T}_{0,N}$, it is easy to check that
	\begin{equation}\label{sr7}\begin{split}
			\mathcal{E}[X_\tau]=&\mathcal{E}[\sum_{u=\tau}^{N-1}\max_{\tau\leq v\leq u}L_v+X_N]\leq \mathcal{E}[\sum_{u=\tau}^{N-1}(\max_{0\leq v\leq u}L_v)\vee 0+X_N]\\
		\leq &\mathcal{E}[\sum_{u=\bar{\tau}}^{N-1}(\max_{0\leq v\leq u}L_v)\vee 0+X_N]=\mathcal{E}[\sum_{u=\bar{\tau}}^{N-1}\max_{0\leq v\leq u}L_v+X_N].
		\end{split}
	\end{equation}
	Noting that for any $\bar{\tau}\leq N-1$ and $u\geq \bar{\tau}$, we have
	\begin{equation}\label{sr8}
		\max_{0\leq v\leq u}L_v=\max_{\bar{\tau}\leq v\leq u}L_v=\max_{\underline{\tau}\leq v\leq u}L_v\geq 0.
	\end{equation}
	Combining Equation \eqref{sr7} and \eqref{sr8} yields that
	\begin{displaymath}
		\mathcal{E}[X_\tau]\leq \mathcal{E}[X_{\bar{\tau}}],\textrm{ for any }\tau\in\mathcal{T}_{0,N}.
	\end{displaymath}
	Therefore, $\bar{\tau}$ is optimal. Besides, the Equation \eqref{sr7} and \eqref{sr8} show that for any $\tau\in \mathcal{T}_{0,N}$,
	\begin{displaymath}
		\mathcal{E}[X_\tau]\leq \mathcal{E}[\sum_{u=\bar{\tau}}^{N-1}\max_{\bar{\tau}\leq v\leq u}L_v+X_N]\leq \mathcal{E}[\sum_{u=\underline{\tau}}^{\bar{\tau}-1}\max_{\underline{\tau}\leq v\leq u}L_v+ \sum_{u=\bar{\tau}}^{N-1}\max_{\underline{\tau}\leq v\leq u}L_v+X_N]=\mathcal{E}[X_{\underline{\tau}}],
	\end{displaymath}
	which implies that $\underline{\tau}$ is also optimal.
	
	Now if $\tau^*$ satisfies \eqref{sr9}, we claim that
	\begin{equation}\label{I-II}
		I:=\sum_{u=\bar{\tau}}^{N-1}\max_{0\leq v\leq u}L_v\leq \sum_{u=\tau^*}^{N-1}\max_{0\leq v\leq u}L_v=:II.
	\end{equation}
	If $\underline{\tau}=N$, then $\tau^*=\bar{\tau}=N$, which means that $I=II=0$. If $\bar{\tau}=\tau^*$, it is obvious that $I=II$. For the case that $\underline{\tau}\leq N-1$ and $\tau^*<\bar{\tau}$, we derive that
	\begin{displaymath}
		\sum_{u=\tau^*}^{\bar{\tau}-1}\max_{0\leq v\leq u}L_v\geq \sum_{u=\tau^*}^{\bar{\tau}-1} L_{\underline{\tau}}\geq 0.
	\end{displaymath}
	Consequently, we obtain that $I\leq II$. Hence the claim holds true. By the condition that $\max_{0\leq v\leq \tau^*}L_v=L_{\tau^*}$ and combining Equations \eqref{sr7}, \eqref{I-II}, it follows that for any $\tau\in\mathcal{T}_{0,N}$,
	\begin{displaymath}
		\mathcal{E}[X_\tau]\leq \mathcal{E}[\sum_{u=\tau^*}^{N-1}\max_{0\leq v\leq u}L_v+X_N]=\mathcal{E}[\sum_{u=\tau^*}^{N-1}\max_{\tau^*\leq v\leq u}L_v+X_N]=\mathcal{E}[X_{\tau^*}].
	\end{displaymath}
	Thus we get the optimality of $\tau^*$.
\end{proof}

\subsubsection{Optimal stopping with $g$-expectation on an infinite horizon}
Here, we present a similar result to Theorem \ref{srt1} for the infinite time case. To this purpose, we first recall some properties of BSDEs with infinite time horizon. Consider the following BSDEs with infinite time horizon:
\begin{equation}\label{infinite}
Y_t=\xi+\int_t^\infty \hat{g}(s,Z_s)ds-\int_t^\infty Z_s d B_s,
\end{equation}
where $\xi\in L^2(\mathcal{F}_\infty)$, which is the collection of all $\mathcal{F}_\infty$-measurable and square-integrable random variables and $\hat{g}$ is a map from $[0,\infty)\times\Omega\times\times \mathbb{R}^d$ onto $\mathbb{R}$ satisfying the following two conditions
\begin{description}
\item[(a)] $\hat{g}(\cdot,z)$ is progressively measurable and $\hat{g}(t,0)=0$ for any $t\in[0,\infty)$;
\item[(b)] There exists a positive deterministic function $u(t)$ such that, for any $z,z'\in\mathbb{R}^d$,
$$|\hat{g}(t,z)-\hat{g}(t,z')|\leq u(t)|z-z'|,\ \ t\in[0,\infty),$$
and $\int_0^\infty u^2(t)dt <\infty$.
\end{description}
By  \cite{Chen}, there exists a unique solution $(Y,Z)\in S^2\times H^2$ satisfying the BSDE \eqref{infinite}, where
\begin{align*}
& S^2:=\{Y|Y_t,0\leq t\leq \infty, \textrm{ is an }\mathcal{F}_t\textrm{-adpated process such that } E[\sup_{t\in[0,\infty]}|Y_t|^2]<\infty\},\\
& H^2:=\{Z|Z_t,0\leq t\leq \infty, \textrm{ is an }\mathcal{F}_t\textrm{-adpated process such that } E[\int_0^\infty|Z_t|^2dt]<\infty\}.
\end{align*}

We define the $\hat{g}$-conditional expectation of $\xi\in L^2(\mathcal{F}_\infty)$ as follows
\begin{displaymath}
\hat{\mathcal{E}}_t[\xi]=Y_t,
\end{displaymath}
where $Y$ is the solution to BSDE \eqref{infinite}. For simplicity, we denote $\hat{\mathcal{E}}_0[\xi]$ by $\hat{\mathcal{E}}[\xi]$. By the results in \cite{HLW}, comparison theorem still holds for $\hat{\mathcal{E}}$. Besides, it is easy to check that $\hat{g}$-expectation also satisfies time-consistency and translation invariance property. Similar analysis to that of  the proof of Theorem \ref{srt1} leads to the following result.
\begin{proposition}\label{srt3}
	Suppose that the adapted process $X=\{X_n\}_{n\in\mathbb{N}}$ with $E[\sup_{n\in\mathbb{N}}|X_n|^2]<\infty$ has the following representation:
	\begin{displaymath}
	X_\tau=\hat{\mathcal{E}}_\tau[\sum_{u=\tau}^{\infty}\sup_{\tau\leq v\leq u}L_v], \textrm{ for any }\tau\in\mathcal{T}_{\infty},
	\end{displaymath}
	where $L=\{L_n\}_{n\in\mathbb{N}}$ is adapted and $\sum_{u=\tau}^{\infty}\sup_{\tau\leq v\leq u}L_v$ is square-integrable for any $\tau\in\mathcal{T}_\infty$. Here, $\mathcal{T}_\infty$ is the collection of all stopping times taking values in $\mathbb{N}$. Then, the level passage times
	\begin{displaymath}
	\underline{\tau}=\inf\{t\geq 0|L_t\geq 0\},\ \ \bar{\tau}=\inf\{t\geq 0|L_t>0\}
	\end{displaymath}
	maximize the expected reward $\hat{\mathcal{E}}[X_\tau]$ over all $\tau\in\mathcal{T}_\infty$.
	
	Furthermore, if the stopping time $\tau^*$ satisfies the following condition
	\begin{displaymath}
	\underline{\tau}\leq \tau^*\leq \bar{\tau}\textrm{ and } \sup_{0\leq v\leq \tau^*}L_v=L_{\tau^*}\textrm{ on }\{\tau^*<\infty\},
	\end{displaymath}
	then $\tau^*$ also maximize $\hat{\mathcal{E}}[X_\tau]$ over all $\tau\in\mathcal{T}_\infty$.
\end{proposition}


\subsection{A variant of Skorokhod's obstacle problem}
Let $f$ satisfy conditions (1) and (2) from Section 2.
Let us now consider the given stochastic sequence $X=\{X_n\}_{n=0}^N$ as a given  obstacle. We wish to find a pair of adapted sequences   $Y=\{Y_n\}_{n=0}^N$  and $\eta=\{\eta_n\}_{n=0}^{N-1}$, with $\eta$ an increasing process, such that
\begin{displaymath}
	Y_t=\mathcal{E}_t[\sum_{u=t}^{N-1}f(u, \eta_u)+X_N],
\end{displaymath}
and such such that $Y$ never falls below the obstacle $X$. It is easy  to check that there are infinitely many processes $Y$ and $\eta$   satisfying the above condition. The goal  is  to find the process $\eta$ which acts in a minimal way, in the sense  that it only increases when necessary (Skorokhod-type condition). This means, if $Y_\tau=X_\tau$, then $\tau$ should be a point of increase of $\eta$, that is,  $\eta_{\tau}>\eta_{\tau-1}$. If $Y_\tau>X_\tau$, the process  $\eta$ should remain the same. We show our result for the case where $f(t,l)=l$. The case of  $f$ satisfying conditions (1) and (2) can be proved similarly.

\begin{remark}
	In order to obtain the uniqueness of the solution to the obstacle problem, we  assume that $\eta_{-1}=-\infty$. Therefore, the initial time $0$ is a point of increase.
\end{remark}

\begin{theorem}\label{srt2}
	Let $X=\{X_n\}_{n=0,1,\cdots,N}$ be an adapted and square-integrable sequence and $L=\{L_t\}_{t=0,1,\cdots,N-1}$ be the unique solution of the following backward equation
	\begin{displaymath}
	X_t=\mathcal{E}_t[\sum_{u=t}^{N-1}\max_{t\leq v\leq u} L_v+X_N].
	\end{displaymath}
	\begin{description}
		\item[(i)] There exists a unique adapted square-integrable  process  $Y=\{Y_n\}_{n=0}^N$ and  a unique adapted square-integrable and nondecreasing process $\eta=\{\eta_n\}_{n=0}^{N-1}$ satisfying
		\begin{displaymath}
			Y_\tau=\mathcal{E}_\tau[\sum_{u=\tau}^{N-1}\eta_u+X_N] ,\ \  \tau\in\mathcal{T}_{0,N},
		\end{displaymath}
		and such that $Y$ dominates $X$, and $Y_\tau=X_\tau$, $P$-a.s. for any point of increase $\tau$ for $\eta$ and $\tau=N$. In fact, $\eta$ has the following representation
		\begin{displaymath}
			\eta_t=\max_{0\leq v\leq t}L_v, \textrm{ for all } t=0,1,\cdots,N-1.
		\end{displaymath}
		\item[(ii)] If the stopping time $\tau^*$ satisfies the following conditions
		\begin{displaymath}
			\underline{\tau}\leq \tau^*\leq \bar{\tau}, \ \ Y_{\tau^*}=X_{\tau^*},
		\end{displaymath}
		where $\underline{\tau}$ and $\bar{\tau}$ are the level passage times
		\begin{displaymath}
			\underline{\tau}:=\min\{v\geq 0|\eta_v\geq 0\}\wedge N \textrm{ and } \bar{\tau}:=\min\{v\geq 0|\eta_v>0\}\wedge N,
		\end{displaymath}
		then $\tau^*$ maximizes $\mathcal{E}[X_\tau]$ over all $\tau\in\mathcal{T}_{0,N}$.
	\end{description}
\end{theorem}

\begin{proof}
	(i) We first show that the process $Y$ associated with the process $\eta$ defined by $L$ dominates $X$ and $Y_\tau=X_\tau$, $P$-a.s. for any point of increase $\tau$ of $\eta$ and $\tau=N$. It is easy to check that $Y_\tau\geq X_\tau$ and $Y_N=X_N$. Now if $\tau$ is a point of increase for $\eta$, we have $\eta_\tau>\eta_{\tau-1}$, which implies that $L_\tau>\max_{0\leq v\leq \tau-1}L_v$. Therefore, for any $u\geq \tau$, it follows that
	\begin{displaymath}
		\max_{0\leq v\leq u}L_v=\max_{\tau\leq v\leq u}L_v,
	\end{displaymath}
	which yields that $Y_\tau=X_\tau$.
	
	We are now in a position to show the uniqueness. Suppose that $\zeta=\{\zeta_t\}_{t=0,1,\cdots,N-1}$ is another adapted, square-integrable and nondecreasing process such that the corresponding adapted process
	\begin{displaymath}
		Z_\tau=\mathcal{E}_\tau[\sum_{u=\tau}^{N-1}\zeta_u+X_N]
	\end{displaymath}
	dominates $X$ with $X_\tau=Z_\tau$ for any point of increase $\tau$ for $\zeta$ and $\tau=N$. For any $\varepsilon>0$, define the following two stopping times
	\begin{displaymath}
		\sigma_\varepsilon=\min\{t\geq 0|\eta_t>\zeta_t+\varepsilon\}\wedge N,\  \ \tau_\varepsilon=\inf\{t\geq \sigma_\varepsilon|\zeta_t\geq \eta_t\}\wedge N.
	\end{displaymath}
	It is easy to check that on the set $\{\sigma_\varepsilon\leq N-1\}$, $\sigma_\varepsilon<\tau_\varepsilon$ and $\sigma_\varepsilon$ is a point of increase for $\eta$. Furthermore, on the set $\{\tau_\varepsilon\leq N-1\}$, $\tau_\varepsilon$ is a point of increase for $\zeta$. By simple calculation, on the set $\{\sigma_\varepsilon\leq N-1\}\cap \{\tau_\varepsilon\leq N-1\}$, we have 
	\begin{align*}
		X_{\sigma_\varepsilon}=Y_{\sigma_\varepsilon}&=\mathcal{E}_{\sigma_\varepsilon}[\sum_{u=\sigma_\varepsilon}^{\tau_\varepsilon-1}\eta_u+\sum_{u=\tau_\varepsilon}^{N-1}\eta_u+X_N]
		>\mathcal{E}_{\sigma_\varepsilon}[\sum_{u=\sigma_\varepsilon}^{\tau_\varepsilon-1}\zeta_u+\mathcal{E}_{\tau_\varepsilon}[\sum_{u=\tau_\varepsilon}^{N-1}\eta_u+X_N]]\\
		&=\mathcal{E}_{\sigma_\varepsilon}[\sum_{u=\sigma_\varepsilon}^{\tau_\varepsilon-1}\zeta_u+Y_{\tau_\varepsilon}]\geq\mathcal{E}_{\sigma_\varepsilon}[\sum_{u=\sigma_\varepsilon}^{\tau_\varepsilon-1}\zeta_u+X_{\tau_\varepsilon}]=\mathcal{E}_{\sigma_\varepsilon}[\sum_{u=\sigma_\varepsilon}^{\tau_\varepsilon-1}\zeta_u+Z_{\tau_\varepsilon}]\\
		&=\mathcal{E}_{\sigma_\varepsilon}[\sum_{u=\sigma_\varepsilon}^{\tau_\varepsilon-1}\zeta_u+\sum_{u=\tau_\varepsilon}^{N-1}\zeta_u+X_N]=Z_{\sigma_\varepsilon}\geq X_{\sigma_\varepsilon}.
	\end{align*}
	On the set $\{\sigma_\varepsilon\leq N-1\}\cap \{\tau_\varepsilon=N\}$, we obtain that
	\begin{displaymath}
			X_{\sigma_\varepsilon}=Y_{\sigma_\varepsilon}=\mathcal{E}_{\sigma_\varepsilon}[\sum_{u=\sigma_\varepsilon}^{N-1}\eta_u+X_N]>\mathcal{E}_{\sigma_\varepsilon}[\sum_{u=\sigma_\varepsilon}^{N-1}\zeta_u+X_N]=Z_{\sigma_\varepsilon}\geq X_{\sigma_\varepsilon}.
	\end{displaymath}
	The contradiction implies that $\sigma_\varepsilon=N$ almost surely, i.e. $\eta_t\leq \zeta_t+\varepsilon$ for any $t=0,1,\cdots,N-1$. Since $\varepsilon$ can be arbitrarily small, this implies that $\eta\leq \zeta$. Consequently, we have $\zeta\leq \eta$. Thus we get the uniqueness.
	
	(ii) 	Since $\eta=\{\eta_t\}=\{\max_{0\leq v\leq t }L_v\}$ is an increasing process, we derive that
	\begin{displaymath}
	Y_t=\mathcal{E}_t[\sum_{u=t}^{N-1}\eta_u+X_N]=\mathcal{E}_t[\sum_{u=t}^{N-1}\max_{t\leq v\leq u}\eta_v+X_N].
	\end{displaymath}
	By Theorem \ref{srt1}, $\bar{\tau}$ maximizes $\mathcal{E}[Y_\tau]$ over all $\tau\in\mathcal{T}_{0,N}$. Noting that on the set $\{\bar{\tau}\leq N-1\}$, $\bar{\tau}$ is a point of increase for $\eta$, we obtain that $X_{\bar{\tau}}=Y_{\bar{\tau}}$, which implies that
	\begin{displaymath}
		\sup_{\tau\in\mathcal{T}_{0,N}} \mathcal{E}[X_\tau]\geq \mathcal{E}[X_{\bar{\tau}}]=\mathcal{E}[Y_{\bar{\tau}}]=\sup_{\tau\in\mathcal{T}_{0,N}}\mathcal{E}[Y_\tau].
	\end{displaymath}
    Since $Y$ dominates $X$, it is obvious that $\sup_{\tau\in\mathcal{T}_{0,N}} \mathcal{E}[X_\tau]\leq \sup_{\tau\in\mathcal{T}_{0,N}} \mathcal{E}[Y_\tau]$. Therefore, the value of the optimal stopping for $X$ equals to the one for $Y$. It is easy to check that $\max_{0\leq v\leq \tau^*}\eta_v=\eta_{\tau^*}$. Theorem \ref{srt1} shows that
    $\mathcal{E}[Y_{\tau^*}]=\sup_{\tau\in\mathcal{T}_{0,N}}\mathcal{E}[Y_\tau]$. We finally get that
    \begin{displaymath}
    	\mathcal{E}[X_{\tau^*}]=\mathcal{E}[Y_{\tau^*}]=\sup_{\tau\in\mathcal{T}_{0,N}}\mathcal{E}[Y_\tau]
    =\sup_{\tau\in\mathcal{T}_{0,N}} \mathcal{E}[X_\tau].
    \end{displaymath}
    The proof is complete.
\end{proof}

We state the result for $f$ satsifying conditions (1) and (2).
\begin{corollary}\label{srt6}
	Assume that the function $f$ satisfies conditions (1) and (2). Let $X=\{X_t\}_{t=0,1,\cdots,N}$ be an adapted and square-integrable sequence and $L=\{L_t\}_{t=0,1,\cdots,N-1}$ be the solution of the following backward equation
	\begin{displaymath}
		X_t=\mathcal{E}_t[\sum_{u=t}^{N-1}f(u,\max_{t\leq v\leq u}L_v)+X_N].
	\end{displaymath}
	Then, there exists a unique adapted square-integrable  process
	 $Y=\{Y_n\}_{n=0}^N$ and  a unique adapted, square-integrable and
	nondecreasing process $\eta=\{\eta_n\}_{n=0}^{N-1}$ satisfying
	\begin{displaymath}
		Y_\tau=\mathcal{E}_\tau[\sum_{u=\tau}^{N-1} f(u,\eta_u)+X_N],\ \ \tau\in\mathcal{T}_{0,N},
	\end{displaymath}
	 such that $Y$ is dominated by $X$ and $Y_\tau=X_\tau$, $P$-a.s. for any point of increase $\tau$ for $\eta$ and $\tau=N$. In fact, $\eta$ has the following representation
		\begin{displaymath}
	\eta_t=\max_{0\leq v\leq t}L_v, \textrm{ for any } t=0,1,\cdots,N-1.
	\end{displaymath}
\end{corollary}

\subsection{Exercising optimally  American puts under Knightian uncertainty}
It is well known that (superhedging) pricing of American options is closely related to  optimal stopping. More precisely, the superhedging price of the American option  corresponds (up to discounting) to the value of an optimal stopping problem and the first time  the discounted Snell envelope hits the discounted payoff process is an  optimal exercise time.  The shortcoming of this approach, when applied to American put options,  is that, in order to derive  optimal exercise times for different strike prices, we need to calculate the associated Snell envelopes first. This would turn into a tedious task as the strike prices may take values in a wide range. One may wonder whether there is a universal process to determine the optimal exercise times simultaneously for different strike prices. With the help of the stochastic representation problem, the answer is affirmative.

In this sub-section, we focus on  American put options with different strike prices $k$, where $k>0$. We place ourselves in an arbitrage-free  market model in discrete time  with two primary assets: a risky asset with price process denoted by $(P_t)_{t=0,1,\cdots,N}$ and a risk-free asset with price process modeled by $((1+r)^{-t})_{t=0,1,\cdots,N},$  where $r$ is a given positive constant, modeling the risk-free interest rate. We consider an agent whose preferences are numerically represented by a utility of the form  of a non-linear expectation $\mathcal{E}$. If an American put option  with strike price $k>0$ on the risky asset is exercised at time $\tau$, then  the pay-off is  $(k-P_\tau)^+$. We consider an  agent who  aims at maximizing the utility of the (discounted) terminal pay-off of the put option over all possible exercise times $\tau$. Thus, the agent aims at solving the following non-linear optimal stopping problem:
$$v= \sup_{\tau\in \mathcal{T}_{0,N}} \mathcal{E}[(1+r)^{-\tau}(k-P_\tau)^+].$$

The following two theorems provide an optimality criterion for  constructing optimal exercise  times for the non-linear optimal stopping  problem  in terms of a universal process $(K_t)$, which is "independent"   of the strike price $k$ of the put option. 
The first theorem gives the existence of the universal process $(K_t)$ via the  non-linear  stochastic representation.  
The universal process $(K_t)$ depends on the discounted price process of the underlying risky asset 
(and hence on the primary assets in the market model)  and on the agent's preferences via  $\mathcal{E}$, but is independent of the strike of the  American put.  

\begin{theorem}\label{srt4}
	Assume that the discounted price process $\{(1+r)^{-t}P_t\}_{t=0,1,\cdots,N}$ is adapted and square-integrable. Then, for any $\tau\in\mathcal{T}_{0,N}$, the discounted price process  admits a unique representation
	\begin{equation}\label{sr10}
		-(1+r)^{-\tau}P_\tau=\mathcal{E}_\tau[\sum_{u=\tau}^{N-1}\frac{r}{1+r}(1+r)^{-u}\max_{\tau\leq v\leq u}(-K_v)+(1+r)^{-N}\max_{\tau\leq v\leq N}(-K_v)]
	\end{equation}
	for some adapted and square-integrable process $K=\{K_t\}_{t=0,1,\cdots,N}$.
	
	For any $k\geq 0$, consider the following two stopping times
	\begin{displaymath}
		\underline{\tau}^k=\min\{0\leq t\leq N|K_t\leq k\}, \ \ \bar{\tau}^k=\min\{0\leq t\leq N|K_t<k\}
	\end{displaymath}
	and the optimal stopping problem
	\begin{equation}\label{sr12}
		V=\sup_{\tau\in \mathcal{T}_{N\cup {+\infty}}}\mathcal{E}[(1+r)^{-\tau}(k-P_\tau)I_{\{\tau\leq N\}}],
	\end{equation}
	where $\mathcal{T}_{N\cup \{+\infty\}}$ is the set of all stopping times taking values in $\{0,1,\cdots,N,+\infty\}$. If a stopping time $\tau^k$ satisfies the following
	\begin{equation}\label{sr11}
		\underline{\tau}^k\leq \tau^k\leq \bar{\tau}^k \textrm{ and } \min_{0\leq v\leq \tau^k}K_v=K_{\tau^k} \textrm{ on } \{\tau^k\leq N\},
	\end{equation}
	then $\tau^k$ is optimal for the problem \eqref{sr12}.
\end{theorem}

\begin{proof}
	The proof will be divided into the following three parts.
	
	Step 1. For any $k\geq 0$, we define the following process $X^k=\{X^k_t\}_{t\in\mathbb{N}\cup\{+\infty\}}$, where
	\begin{displaymath}
		X_t^k=(1+r)^{-t}(k-P_{t\wedge N}).
	\end{displaymath}
	Consider the following optimal stopping problem
	\begin{equation}\label{sr13}
		V'=\sup_{\tau\in\mathcal{T}_\infty}\hat{\mathcal{E}}[X_\tau^k],
	\end{equation}
where $\hat{\mathcal{E}}[\cdot]$ is the $\hat{g}$-expectation for the infinite time case with
$$\hat{g}(t,z)=g(t,z)I_{\{t\leq N\}}+e^{-t}z I_{\{t>N\}}.$$
Clearly, for any $\mathcal{F}_N$-measurable and square-integrable random variable $\xi$, we have $\hat{\mathcal{E}}[\xi]=\mathcal{E}[\xi]$.
	We claim that $V=V'$ and the optimal stopping times for \eqref{sr12} and \eqref{sr13} are the same. Since $r>0$, we derive that if $\tau^*$ is optimal for \eqref{sr13}, then $\tau^*$ takes values in $\{0,1,\cdots,N,+\infty\}$. Therefore, we have
	\begin{align*}
		\sup_{\tau\in \mathcal{T}_\infty}\hat{\mathcal{E}}[X_\tau^k]&=\hat{\mathcal{E}}[X^k_{\tau^*}]=\hat{\mathcal{E}}[(1+r)^{-\tau^*}(k-P_{\tau^*\wedge N})]\\
		&=\mathcal{E}[(1+r)^{-\tau^*}(k-P_{\tau^*})I_{\{\tau^*\leq N\}}]\\
		&\leq \sup_{\tau\in \mathcal{T}_{N\cup {+\infty}}}\mathcal{E}[(1+r)^{-\tau}(k-P_\tau)I_{\{\tau\leq N\}}].
	\end{align*}
	Besides, for any $\tau\in\mathcal{T}_{N\cup\{+\infty\}}$, it is easy to check that
	\begin{displaymath}
		\mathcal{E}[(1+r)^{-\tau}(k-P_\tau)I_{\{\tau\leq N\}}]=\mathcal{E}[(1+r)^{-\tau}(k-P_{\tau\wedge N})]=\hat{\mathcal{E}}[X_\tau^k].
	\end{displaymath}
	It follows that
	\begin{displaymath}
		\sup_{\tau\in \mathcal{T}_{N\cup {+\infty}}}\mathcal{E}[(1+r)^{-\tau}(k-P_\tau)I_{\{\tau\leq N\}}]\leq \sup_{\tau\in \mathcal{T}_{N\cup {+\infty}}}\hat{\mathcal{E}}[X_\tau^k]\leq \sup_{\tau\in \mathcal{T}_\infty}\hat{\mathcal{E}}[X_\tau^k].
	\end{displaymath}
	Consequently, we obtain that $V=V'$ and the optimal stopping problems \eqref{sr12} and \eqref{sr13} have the same set of maximizers.
	
	Step 2. For any $t\in \mathbb{N}$, set
	\begin{displaymath}
		L_t^k=k-K_{t\wedge N}.
	\end{displaymath}
	We claim that
	\begin{displaymath}
		X_\tau^k=\hat{\mathcal{E}}_\tau[\sum_{u=\tau}^\infty \frac{r}{1+r}(1+r)^{-u}\sup_{\tau\leq v\leq u}L_v^k].
	\end{displaymath}
	Indeed, by simple calculation, we obtain that
	\begin{align*}
		&\hat{\mathcal{E}}_\tau[\sum_{u=\tau}^\infty \frac{r}{1+r}(1+r)^{-u}\sup_{\tau\leq v\leq u}L_v^k]\\
		=&\hat{\mathcal{E}}_\tau[\sum_{u=\tau}^\infty \frac{r}{1+r}(1+r)^{-u}\sup_{\tau\leq v\leq u}(k-K_{v\wedge N})]\\
		=&k(1+r)^{-\tau}+\hat{\mathcal{E}}_\tau[\sum_{u=\tau}^\infty \frac{r}{1+r}(1+r)^{-u}\sup_{\tau\leq v\leq u}(-K_{v\wedge N})]\\
		=&k(1+r)^{-\tau} +\hat{\mathcal{E}}_\tau[\sum_{u=\tau\wedge N}^{N-1} \frac{r}{1+r}(1+r)^{-u}\sup_{\tau\wedge N \leq v\leq u}(-K_v)+\sum_{u=\tau\vee N}^\infty \frac{r}{1+r}(1+r)^{-u}\sup_{\tau\wedge N\leq v\leq N}(-K_v)]\\
		=&k(1+r)^{-\tau}+\hat{\mathcal{E}}_\tau[\sum_{u=\tau\wedge N}^{N-1} \frac{r}{1+r}(1+r)^{-u}\sup_{\tau\wedge N \leq v\leq u}(-K_v)+(1+r)^{-\tau \vee N}\sup_{\tau\wedge N\leq v\leq N}(-K_v)].
	\end{align*}
	Denote by the second term in the last equality by $I$. Then on the set $\{\tau\leq N-1\}$, by Equation \eqref{sr10}, we have $I=-(1+r)^{-\tau}P_\tau$. Besides, on the set $\{\tau\geq N\}$, again by \eqref{sr10}, we derive that $I=(1+r)^{-\tau}(-K_N)=-(1+r)^{-\tau}P_N$. The above analysis shows that
	\begin{displaymath}
		\hat{\mathcal{E}}_\tau[\sum_{u=\tau}^\infty \frac{r}{1+r}(1+r)^{-u}\sup_{\tau\leq v\leq u}L_v^k]=k(1+r)^{-\tau}+I=(1+r)^{-\tau}(k-P_{\tau\wedge N})=X_\tau^k.
	\end{displaymath}
	
	Step 3. By Proposition \ref{srt3}, if $\tau^k$ satisfies the following condition
	\begin{equation}\label{sr14}
		\underline{\sigma}^k\leq \tau^k\leq \bar{\sigma}^k\textrm{ and } \sup_{0\leq v\leq \tau^k}L_v^k=L^k_{\tau^k} \textrm{ on } \{\tau^k<+\infty\},
	\end{equation}
	where $\underline{\sigma}^k=\min\{t\geq 0|L_t^k\geq 0\}$ and $\bar{\sigma}^k=\min\{t\geq 0|L_t^k>0\}$, then $\tau^k$ is optimal for the problem \eqref{sr13}. By Step 1, we know that $\{\tau^k<\infty\}=\{\tau^k\leq N\}$. By the definition of $L^k$, it is easy to check that $\underline{\sigma}^k=\underline{\tau}^k$ and $\bar{\sigma}^k=\bar{\tau}^k$ and all these stopping times belong to $\mathcal{T}_{N\cup \{+\infty\}}$. It follows that condition \eqref{sr14} is equivalent to condition \eqref{sr11}. Finally, we conclude that for any stopping time $\tau^k$ satisfying condition \eqref{sr11}, $\tau^k$ is optimal for problem \eqref{sr13}, hence optimal for problem \eqref{sr12} by Step 1.
\end{proof}

\begin{theorem}
	For any $\tau\in\mathcal{T}_{0,N}$, the solution $K$ of Equation \eqref{sr10} satisfies $K_\tau\geq P_\tau$, a.s. Besides, the restriction $\tau^k\wedge N$ of any optimal stopping time $\tau^k$ defined by Theorem \ref{srt4} is also optimal for the following problem
	\begin{displaymath}
		v=\sup_{\tau\in \mathcal{T}_{0,N}} \mathcal{E}[(1+r)^{-\tau}(k-P_\tau)^+].
	\end{displaymath}
\end{theorem}

\begin{proof}
	For any $\tau\in\mathcal{T}_{0,N}$, it is easy to check that
	\begin{align*}
		-(1+r)^{-\tau}P_\tau=&\mathcal{E}_\tau[\sum_{u=\tau}^{N-1}\frac{r}{1+r}(1+r)^{-u}\max_{\tau\leq v\leq u}(-K_v)+(1+r)^{-N}\max_{\tau\leq v\leq N}(-K_v)]\\
		\geq &\mathcal{E}_\tau[\sum_{u=\tau}^{N-1}\frac{r}{1+r}(1+r)^{-u}(-K_\tau)+(1+r)^{-N}(-K_\tau)]\\
		=&-(1+r)^{-\tau}K_\tau,
	\end{align*}
	which implies that $P_\tau\leq K_\tau$. We claim that on the set $\{\tau^k\leq N\}$, $K_{\tau^k}\leq k$. Otherwise, $P(\{\tau^k\leq N\}\cap \{K_{\tau^k}>k\})>0$. Since $\underline{\tau}^k\leq \tau^k\leq N$, we have $K_{\underline{\tau}^k}\leq k$. Therefore, on the set $\{\tau^k\leq N\}\cap \{K_{\tau^k}>k\}$, we obtain that
	\begin{displaymath}
		\min_{0\leq v\leq \tau^k}K_v\leq k\neq K_{\tau^k},
	\end{displaymath}
	which leads to a contradiction. It follows that $P_{\tau^k}\leq K_{\tau^k}\leq k$ on the set $\{\tau^k\leq N\}$. Thus,
	\begin{displaymath}
		\mathcal{E}[(1+r)^{-\tau^k}(k-P_{\tau^k})I_{\{\tau^k\leq N\}}]=\mathcal{E}[(1+r)^{-\tau^k\wedge N}(k-P_{\tau^k\wedge  N})^+]
	\end{displaymath}
	and then $\tau^k\wedge N$ maximizes $\mathcal{E}[(1+r)^{-\tau}(k-P_\tau)^+]$ over all $\tau\in\mathcal{T}_{0,N}$.
\end{proof}

\section*{Acknowledgments}
\noindent Financial support by the German Research Foundation (DFG) through the Collaborative Research Centre 1283 ``Taming uncertainty and profiting from randomness and low regularity in analysis, stochastics and their applications'' is gratefully acknowledged.

\end{document}